\documentclass[11pt,a4paper]{article}
\usepackage[T1]{fontenc}
\usepackage[utf8]{inputenc}
\usepackage[english]{babel}
\usepackage{amsmath, amsthm, amssymb} 
\usepackage{mathtools} 
\usepackage{fullpage}
\usepackage{xcolor}
\usepackage{graphicx}
\usepackage{tikz}
\usepackage{bm}
\usepackage{float, caption, subcaption}
\usepackage{microtype} 
\usepackage{hyperref}

\usetikzlibrary{cd, positioning}
\newenvironment{tikzccd}{\begin{center}\begin{tikzcd}}{\end{tikzcd}\end{center}} 
\newcommand{\edge}[1]{\arrow[dash,#1]} 

\microtypesetup{babel=true,tracking=false,kerning=true} 
\hypersetup{ 
    colorlinks,
    linkcolor={red!50!black},
    citecolor={blue!50!black},
    urlcolor={blue!80!black}
}

\theoremstyle{plain}
\newtheorem{thm}{Theorem}[section]
\newtheorem{prop}[thm]{Proposition}
\newtheorem{lem}[thm]{Lemma}

\theoremstyle{remark}
\newtheorem*{rmk}{Remark}
\theoremstyle{definition}
\newtheorem*{claim}{Claim}
\newtheorem{defn}[thm]{Definition}

\makeatletter
\let\save@mathaccent\mathaccent
\newcommand*\if@single[3]{
	\setbox0\hbox{${\mathaccent"0362{#1}}^H$}%
	\setbox2\hbox{${\mathaccent"0362{\kern0pt#1}}^H$}%
	\ifdim\ht0=\ht2 #3\else #2\fi }
\newcommand*\rel@kern[1]{\kern#1\dimexpr\macc@kerna}
\newcommand*\widebar[1]{\@ifnextchar^{{\wide@bar{#1}{0}}}{\wide@bar{#1}{1}}}
\newcommand*\wide@bar[2]{\if@single{#1}{\wide@bar@{#1}{#2}{1}}{\wide@bar@{#1}{#2}{2}}}
\newcommand*\wide@bar@[3]{
	\begingroup
	\def\mathaccent##1##2{
		\let\mathaccent\save@mathaccent
		\if#32 \let\macc@nucleus\first@char \fi
		\setbox\z@\hbox{$\macc@style{\macc@nucleus}_{}$}
		\setbox\tw@\hbox{$\macc@style{\macc@nucleus}{}_{}$}
		\dimen@\wd\tw@ \advance\dimen@-\wd\z@ \divide\dimen@ 3 \@tempdima\wd\tw@ \advance\@tempdima-\scriptspace \divide\@tempdima 10 \advance\dimen@-\@tempdima \ifdim\dimen@>\z@ \dimen@0pt \fi \rel@kern{0.6}\kern-\dimen@
		\if#31 \overline{\rel@kern{-0.6}\kern\dimen@\macc@nucleus\rel@kern{0.4}\kern\dimen@} \advance\dimen@0.4\dimexpr\macc@kerna \let\final@kern#2 \ifdim\dimen@<\z@ \let\final@kern1 \fi
		\if \final@kern1 \kern-\dimen@ \fi
		\else \overline{\rel@kern{-0.6}\kern\dimen@#1} \fi }
	\macc@depth\@ne	\let\math@bgroup\@empty \let\math@egroup\macc@set@skewchar 	\mathsurround\z@ \frozen@everymath{\mathgroup\macc@group\relax} 	\macc@set@skewchar\relax \let\mathaccentV\macc@nested@a	\if#31 \macc@nested@a\relax111{#1} \else \def\gobble@till@marker##1\endmarker{} \futurelet\first@char\gobble@till@marker#1\endmarker \ifcat\noexpand\first@char A\else \def\first@char{} \fi \macc@nested@a\relax111{\first@char} \fi
	\endgroup }
\makeatother

\makeatletter
\def\blfootnote{\xdef\@thefnmark{}\@footnotetext}
\makeatother

\DeclareMathOperator{\rk}{rk}

\DeclareMathOperator{\Exp}{\ensuremath{\mathbb{E}}}

\newcommand{\calA}{\ensuremath{\mathcal{A}}}
\newcommand{\calB}{\ensuremath{\mathcal{B}}}
\newcommand{\calC}{\ensuremath{\mathcal{C}}}
\newcommand{\calD}{\ensuremath{\mathcal{D}}}

\newcommand{\calI}{\ensuremath{\mathcal{I}}}
\newcommand{\calS}{\ensuremath{\mathcal{S}}}

\newcommand{\N}{\ensuremath{\mathbb{N}}} 
\newcommand{\Z}{\ensuremath{\mathbb{Z}}} 

\DeclarePairedDelimiter{\floor}{\lfloor}{\rfloor}
\DeclarePairedDelimiter{\ceil}{\lceil}{\rceil}

\title{Sidon set systems}
\author{
    Javier Cilleruelo\thanks{We want to dedicate this paper to the memory of our great colleague and friend Javier
    Cilleruelo, who sadly passed away while this project was being completed. Sidon problems were
    one of his mathematical passions, and he enjoyed a lot the new perspective on these problems
    that is treated in this paper.}
    \and 
    Oriol Serra
    \and
    Maximilian W\"otzel
    }
\date{}

\begin{document}
\maketitle

\begin{abstract}
A family $\calA$ of $k$-subsets of $\{1,2,\dots, N\}$ is a Sidon system if the sumsets $A+B$, $A,B\in \calA$ are pairwise distinct. 
We show that the largest cardinality $F_k(N)$ of a Sidon system of $k$-subsets of $[N]$ satisfies $F_k(N)\le {N-1\choose k-1}+N-k$ and the asymptotic lower bound $F_k(N)=\Omega_k(N^{k-1})$. 
More precise bounds on $F_k(N)$ are obtained for $k\le 3$.  
We also obtain the threshold probability for a random system to be Sidon for $k \geq 2$. 
\end{abstract}

\section{Introduction and Results}\label{sec:intro}

A set $A$ of integers is a Sidon set if the twofold sums of elements in $A$ are pairwise distinct. 
The study of Sidon sets of maximum cardinality is a classic topic in additive number theory; see e.g. the survey by O'Bryant~\cite{OBryant04}. 
The notion can be extended in a natural way to set systems. 

\begin{defn}
Let $\calA = \{A_i : i \in I,\, A_i \subseteq \Z\}$ be a family of subsets of the integers.
We say that $\calA$ is a \emph{Sidon system} if 
\begin{equation*}
A_i + A_j = A_{i'} + A_{j'} \implies \{i,j\} = \{i',j'\}.
\end{equation*}
\end{defn}   

A central problem in this setting is to estimate the maximum cardinality of a Sidon system. 
We restrict ourselves to uniform set systems of $k$-subsets of the integer interval $[N]=\{1,2,\dots ,N\}$.

\begin{defn}
Given integers $N>k \geq 1$, we denote by $F_k(N)$ the largest cardinality of a Sidon system $\calA \subseteq \binom{[N]}{k}$.
\end{defn}

The case $k=1$ corresponds to classical Sidon sets and it is well-known that $F_1(N)~\sim~N^{1/2}$. 
We address here the problem for $k \geq 2$. First we give an upper bound for $F_k(N)$. 

\begin{thm}\label{thm:upper}
For $2\le k < N$ we have
\begin{equation*}
F_k(N) \leq \binom{N-1}{k-1} + N - k.
\end{equation*}
\end{thm}

This upper bound is tight for $k = 2$, that is, $F_2(N) = 2N - 3$.
We believe that it is asymptotically sharp for any $k \geq 2$, but we are only able to prove this for $k = 2$ and $k = 3$.
For $k \geq 4$, we prove that $N^{k-1}$ is the right order of magnitude. 

\begin{thm}\label{thm:lower}
\begin{align*}
F_2(N) &= 2N - 3\\
F_3(N) &\geq N^2/2 - O(N)\\
F_k(N) &= \Omega_k(N^{k-1}), \: k \geq 4.
\end{align*}
\end{thm}

We also consider the problem of determining the size of a \emph{typical} Sidon system.
We consider the following model for random uniform families of subsets of $[N]$. 

\begin{defn}\label{defn:distribution}
Let $N>k$ be positive integers and $0<p\leq 1$. 
A random set system $\calA$ in $\binom{[N]}{k}$ is obtained by choosing independently every set $A\in \binom{[N]}{k}$ with probability $p$.  
We write $\calA \in \calS(N,k,p)$ to denote that $\calA$ is a random system in this model.
\end{defn}

For $k=1$ the above definition corresponds to the binomial model of random subsets. 
Godbole, Janson, Locantore and Rapoport~\cite{GJLR99} showed, among more general results, that $N^{-3/4}$ is the threshold probability for a random set in $[N]$ to be a Sidon set.
In the case of random systems we obtain the threshold probability for all $k\geq 2$. 

\begin{thm}\label{thm:randomsidon}
Let $\calA\in \calS(N,k,p)$ be a random system.
Then, for any $k\geq 2$,
\begin{equation*}
\lim_{N\to\infty}\Pr(\text{$\calA$ is Sidon}) = 
\begin{cases} 
1, &\text{if } p = o(N^{-(2k+1)/4}) \\ 
0, &\text{if } p = \omega(N^{-(2k+1)/4})
\end{cases}.
\end{equation*}
\end{thm}

We note that the value of the threshold probability in Theorem~\ref{thm:randomsidon} fits with the results in~\cite{GJLR99} for Sidon sets (i.e., $k=1$). 
This is in contrast with the results obtained in Theorems~\ref{thm:upper} and~\ref{thm:lower}.  
In particular, the relative gap between the random and the extremal setting is slightly smaller in the case of Sidon systems, although not by the amount that one may expect, given that a priori, equality of sumsets seems to be a stronger condition than equality of sums of numbers.

We also note that the extension of additive problems in the integers or in additive groups to the monoid of sumsets has been considered in the literature. 
For instance, Cilleruelo, Hamidoune and Serra~\cite{CHS10} proved analogues of the Cauchy--Davenport and Vosper theorems in this setting.
An important question related to the current work is whether a certain set can be expressed as a sumset in multiple ways.
Alon~\cite{Alon07} used probabilistic arguments combined with spectral techniques to improve earlier bounds by Green~\cite{Green05} on the maximal cardinality of subsets of a cyclic group of prime order that cannot be expressed as a sumset.
Fan and Tringali~\cite{FT17} use tools from factorization theory to give (among other results) necessary and sufficient conditions for  certain subsets of integers to be written as sumsets in more than one way.
Selfridge and Straus~\cite{SS58} showed that the representation function $r_A(n)=|\{(a,a')\in A\times A: n=a+a'\}|$ of a subset $A$ in a field of characteristic zero determines the set.
They also considered the general case of $h$-fold sumsets $hA$ and gave necessary conditions when the representation function of this sumset completely determines the set $A$.
These results were later generalized by Gordon, Fraenkel and Straus~\cite{GFS62} to torsion free abelian groups.
For a more detailed look on these problems, see also the recent survey by Fomin~\cite{Fomin17}.
In contrast to these results, in the asymmetric case, the representation function does in general not determine the summands, even in the case of twofold sumsets $A+B$. 

The paper is organized as follows. 
First, in Section~\ref{sec:not}, we introduce some needed definitions and notation that will be used in the remainder of the article.
In Section~\ref{sec:ub}, we prove the upper bound for $F_k(N)$ in Theorem~\ref{thm:upper}. 
Section~\ref{sec:lower} gives constructions of Sidon systems based on analogous constructions for Sidon sets which prove the lower bounds in Theorem~\ref{thm:lower}. 
Section~\ref{sec:random} discusses the threshold probability for a random system to be Sidon. 
We conclude the paper with final remarks in Section~\ref{sec:final}.

\section{Notation}\label{sec:not}

For integers $N>k$, let $\binom{[N]}{k}_0$ denote the family of $k$-subsets of $\{0,1,\dots,N\}$ that all contain $0$.
If $A$ is a $k$-subset of $[N]$, we can always write it in the form $\min(A)+A'$, where $A'\in \binom{[N]}{k}_0$.
If not otherwise stated, we will often use the following notational convention: for a set $A$ we denote by the lower case its minimum element, $a=\min (A)$, and we denote by a dashed letter the translation $A'=A-a$. 
In this context, the set $A'$ will be referred to as the \emph{distance set} of $A$.

Let $A,B,U,V$ be $k$-subsets of $[N]$.
We observe that the equation
\begin{equation}\label{eq:equiva}
A+B=U+V
\end{equation}
is equivalent to 
\begin{equation}\label{eq:equivb}
a+b=u+v \quad\text{and}\quad A'+B'=U'+V',
\end{equation}
so we can often restrict ourselves to elements of $\binom{[N]}{k}_0$ and consider translations separately.

\

Finally, we denote by $\preceq$ the lexicographic order of $k$-subsets of $[N]$, namely, $A\preceq B$ if and only if $\min(A \Delta B) \in A$ or $A=B$, where $A\Delta B$ denotes the symmetric difference of $A$ and $B$. 
Similarly, we define $\prec$ such that $A\prec B$ if and only if $A\preceq B$ and $A\neq B$.
We also use $\preceq$ (resp. $\prec $) to denote the induced lexicographic order on tuples of $k$-subsets.

\section{An Upper Bound for Sidon Systems}\label{sec:ub}

\begin{proof}[Proof of Theorem~\ref{thm:upper}] 
Let $\calA$ be a Sidon system of $k$-subsets in $[N]$. 
For each set $A$ in $\binom{[N-1]}{k-1}$ define \[\calA(A)=\{x\in [N] : x+(A\cup\{0\})\in\calA\}.\]
We have
\begin{equation}\label{eq:a}
|\calA|=\sum_{A\in {[N-1]\choose k-1}} |\calA (A)|.
\end{equation}
Denoting by $Z_+=Z\cap \N$ the set of positive numbers in a set $Z$ of integers, we clearly have
\begin{equation}\label{eq:ax}
|\calA(A)| \le |(\calA(A)-\calA(A))_+|+1.
\end{equation}
We observe that if $A\neq B$ are sets in $\binom{[N-1]}{k-1}$, then 
\begin{equation}\label{eq:axaydisj}
\left(\calA(A)-\calA(A)\right)_+ \cap \left(\calA(B)-\calA(B)\right)_+ = \emptyset.
\end{equation}
Indeed, suppose there are $x'> x$ in $\calA(A)$ and $y>y'$ in $\calA(B)$ such that $x'-x=y-y'$.
Then this implies $(x+A)+(y+B)=(x'+A)+(y'+B)$, which is a violation to the Sidon property of $\calA$.
Equation~\ref{eq:axaydisj} directly implies that
\begin{equation}\label{eq:axay}
\left|\bigcup_{A\in {[N-1]\choose k-1}} (\calA(A)-\calA(A))_+\right|=\sum_{A\in {[N-1]\choose k-1}}|(\calA (A)-\calA(A))_+|.
\end{equation}
Since $\max(A)\geq k-1$ for any set $A$ in $\binom{[N-1]}{k-1}$ and $x+\max(A)\leq N$ for any $x\in\calA(A)$, we clearly have $\calA(A)\subseteq [N-k+1]$, and hence $(\calA(A)-\calA(A))_+ \subseteq [N-k]$.
Together with \eqref{eq:axaydisj}, this implies
\begin{equation}\label{eq:axall}
\left|\bigcup_{A\in {[N]\choose k-1}} (\calA(A)-\calA(A))_+\right|\le N-k.
\end{equation}
Combining Equations~\ref{eq:a}, \ref{eq:ax}, \ref{eq:axay}, and~\ref{eq:axall}, we conclude
\begin{align*}
|\calA| 
&\leq \binom{N-1}{k-1} + \sum_{A\in \binom{[N-1]}{k-1}} |(\calA(A)-\calA(A))_+| \\
&= \binom{N-1}{k-1} + \left|\bigcup_{A\in \binom{[N-1]}{k-1}} (\calA(A)-\calA(A))_+\right| \\
&\leq \binom{N-1}{k-1} + N - k.
\end{align*}
\end{proof}

\section{A Lower Bound for Sidon Systems}\label{sec:lower}

We will consider the cases $k=2$, $k=3$ and $k\geq 4$ separately. 

\begin{prop}\label{prop:k2}
$F_2(N) \geq 2N - 3$.
\end{prop} 

\begin{proof} 
It is easy to check that the family \[\calA = \{\{1,1+i\} : i = 1, \dots, N-1\} \cup \{\{N-i,N\} : i = 1, \dots, N-2\}\] is a Sidon system which satisfies the inequality.
\end{proof} 

\begin{prop}\label{prop:k3}
$F_3(N) \geq N^2/2 - O(N)$.
\end{prop}

\begin{proof} 
Let $A,B$ be elements of $\binom{[N-1]}{3}_0$, write $A=\{0<a_1<a_2\}$ and $B=\{0<b_1<b_2\}$. 
Without loss of generality, we can assume $a_1\le b_1$. 

If all sums in $A+B$ are distinct, then the equations 
\begin{align*}
a_1 &= \min \left((A+B)\setminus\{0\}\right) \\
a_2+b_2 &= \max(A+B) \\
3(a_1+b_1+a_2+b_2) &= \sum_{z\in A+B} z
\end{align*}
determine $a_1$, $b_1$, and $a_2+b_2$ from $A+B$. 
Moreover, $\{a_1+b_2,a_2+b_1\}$ are the second and third largest elements, say $s'<s$, in $A+B$. 
If $a_1+b_2=s$ then $b_2=s-a_1$ and $a_2=\max(A+B)-b_2$ are the two points distinct from $\{0,a_1,b_1,a_1+b_1,s,s',a_2+b_2\}$ in  $A+B$, otherwise these points are $a_2=s-b_1$ and $b_2=\max(A+B)-a_2$, and only one of these two possibilities can occur.
Therefore, if $|A+B|=9$ then the sumset $A+B$ determines the sets $A,B$. 

If $|A+B|=5$ then both $A$ and $B$ are arithmetic progressions with the same common difference, that is, $A=B$ are dilations of $\{0,1,2\}$.

If $6\le |A+B|\le 8$ then a case analysis, which is detailed in Appendix~\ref{sec:3solutions}, shows that the only possibilities for the pair $A,B$ are dilations of the sets in following list:
\begin{alignat*}{2}
\{0,1,2\}+\{0,2,5\}&=\{0,1,2\}+\{0,3,5\}&&=\{0,1,2,3,4,5,6,7\}\\
\{0,1,3\}+\{0,1,5\}&=\{0,1,4\}+\{0,2,4\}&&=\{0,1,2,3,4,5,6,8\}\\
\{0,1,3\}+\{0,4,6\}&=\{0,1,4\}+\{0,3,5\}&&=\{0,1,3,4,5,6,7,9\}\\
\{0,2,3\}+\{0,4,5\}&=\{0,2,4\}+\{0,3,4\}&&=\{0,2,3,4,5,6,7,8\}\\
\{0,2,3\}+\{0,2,6\}&=\{0,2,5\}+\{0,3,4\}&&=\{0,2,3,4,5,6,8,9\}\\
\{0,1,2\}+\{0,1,4\}&=\{0,1,2\}+\{0,2,4\}&&=\{0,1,2,3,4,5,6\}\\
\{0,1,2\}+\{0,3,4\}&=\{0,1,3\}+\{0,2,3\}&&=\{0,1,2,3,4,5,6\}\\
\{0,1,3\}+\{0,1,4\}&=\{0,1,3\}+\{0,2,4\}&&=\{0,1,2,3,4,5,7\}\\
\{0,2,3\}+\{0,2,4\}&=\{0,2,3\}+\{0,3,4\}&&=\{0,2,3,4,5,6,7\}\\
\{0,1,2\}+\{0,1,3\}&=\{0,1,2\}+\{0,2,3\}&&=\{0,1,2,3,4,5\}.
\end{alignat*}
All the above solutions involve either dilations of $\{0,1,2\}$ or $\{0,1,3\}$, so by eliminating these, one obtains a Sidon system with $N^2/2-O(N)$ sets. 
More precisely, let \[\calB=\{\lambda\cdot \{0,1,2\},\, \lambda\cdot \{0,1,3\} : 1\le \lambda\le \floor{N/2}\}.\]
Then the set \[\calA = \left\{1+A : A \in \binom{[N-1]}{3}_0 \setminus \calB\right\}\] is a Sidon system with \[|\calA| \geq \binom{N-1}{2} - \frac{5}{6}N.\]
This completes the proof.
\end{proof}

For $k\ge 4$ we give a somewhat more involved construction of a large Sidon system.

\begin{prop}\label{prop:k4}
For $k\geq 4$, it holds that \[F_k(N) =\Omega_k(N^{k-1}).\]
\end{prop}

\begin{proof}
Let $k\geq 3$ be an integer, and let $A = \{a_0, a_1, \dots, a_{k-1}\}$ be a Sidon set such that \[0=a_0 < a_1 < \dots < a_{k-1} \leq 2k^2,\] which is well known to exist.
We can assume that $N \geq 2(2k^2+1)$. 
For $i=1,\dotsc,k-1$, denote by $I_i$ the interval \[I_i=\left(\frac{N}{a_{k-1}+1}\cdot [a_i,a_i+1/2)\right)\cap \N,\] with cardinality \[|I_i|=\floor*{\frac{N}{2(a_{k-1}+1)}}.\]
Let $I_0=\{0\}$.
Since $A$ is a Sidon set, for all $i,j,i',j'\in [0,k-1]$ we have \[(I_i+ I_j)\cap (I_{i'}+I_{j'})=\emptyset,\] unless $\{i,j\}=\{i',j'\}$. 
Consider the family \[\calB = \{\{b_0,\dots,b_{k-1}\} : b_i\in I_i\} \subseteq \binom{[N-1]}{k}_0,\] which has cardinality \[|\calB| \ge \floor*{\frac{N}{2(a_{k-1}+1)}}^{k-1} \geq \left(\frac{N}{2(2k^2+1)}\right)^{k-1}.\]
It remains to be proven that $\calB$ is a Sidon system.

Let $U=\{0<u_1<\dots<u_{k-1}\}$ and $V=\{0<v_1<\dots<v_{k-1}\}$ be sets in $\calB$. 
We will show that $U+V$ determines univocally the sets $U,V$. 
We have \[U+V\subseteq \bigcup_{0\le i\le j\le k-1} (I_i+I_j).\]
Since the intervals $I_i+I_j$ are pairwise disjoint, each of them contains at most the two elements \[u_i+v_j,\, u_j+v_i\in I_i+I_j\] from $U+V$.
In particular, for any $i\in[1,k-1]$, the set $(U+V)\cap I_i$ contains exactly one element if and only if $x_i=y_i$, and clearly, if this happens for all $i$, then $U=V$ and the set is determined univocally by the elements in $(U+V)\cap I_i$.
So assume that this is not the case, and let $i_0$ denote the least positive integer such that $(U+V)\cap I_{i_0}$ contains two elements.
Hence for all $0\leq i < i_0$, the elements $u_i=v_i$ are determined by $U+V$.
Without loss of generality, we can assume that $U \preceq V$.
Therefore \[u_{i_0}=\min((U+V)\cap I_{i_0}) < \max((U+V)\cap I_{i_0}) = v_{i_0},\] and so both $u_{i_0}$ and $v_{i_0}$ are determined by $U+V$.
Let $i>i_0$. 
We have \[(U+V)\cap I_i=\{u_i,v_i\}\;\text{and}\; (U+V)\cap (I_{i_0}+I_i)=\{u_{i_0}+v_i,u_i+v_{i_0}\}.\]
So by subtracting $u_{i_0}$ and $v_{i_0}$ from the elements in $(U+V)\cap (I_{i_0}+I_i)$, we can determine $u_i$ and $v_i$.
This shows that every sumset in $\calB+\calB$ can be written uniquely as a sum of two sets from $\calB$, and so $\calB$ is a Sidon system.

We can now shift every set in $\calB$ by $1$ such that the resulting family $\calB'$ is a Sidon system of $k$-subsets of $[N]$ with the same cardinality as $\calB$.
\end{proof}

\begin{proof}[Proof of Theorem~\ref{thm:lower}]
Propositions~\ref{prop:k2}, \ref{prop:k3}, and~\ref{prop:k4} provide the lower bounds, while Theorem~\ref{thm:upper} with $k=2$ gives the upper bound for the first equation.
\end{proof}

\section{Random Sidon Systems}\label{sec:random}

In this section we analyze the threshold probability of random families of  $k$-subsets of $[N]$ for the property of being a Sidon system. 

For some fixed $k\geq 2$, let \[\calB = \left\{A=(A_1,A_2) : A_1,A_2 \in \binom{[N]}{k},\; A_1 \preceq A_2\right\},\] denote the family of ordered pairs of $k$-subsets of $[N]$, and let \[\calC = \left\{(A,B) \in \calB\times\calB : A \prec  B,\; A_1+A_2=B_1+B_2\right\}\] be the family of distinct ordered pairs of elements in $\calB$ which violate the Sidon property.

For $(A,B)\in\calC$, let $I_{A,B}$ denote the indicator variable that  $\{A_1,A_2,B_1,B_2\}$ belongs to the random system $\calA$, and define \[X = \sum_{(A,B)\in\calC}I_{A,B}.\]
Therefore, \[\Pr(\calA \text{ is a Sidon system})=\Pr(X=0).\]
Finally, for $2 \leq \ell \leq 4$, define \[\calC(\ell) = \left\{(A,B)\in \calC : |\{A_1,A_2,B_1,B_2\}| = \ell\right\}.\]
Every $k$-subset of $[N]$ is contained in $\calA$ independently with probability $p$, hence for each $(A,B)\in\calC(\ell)$, we have that \[\Exp (I_{A,B})=p^\ell.\]
We will first prove the $1$-statement in Theorem~\ref{thm:randomsidon}, and begin by giving upper bounds for the cardinalities of the $\calC(\ell)$.

\begin{lem}\label{lem:CLupper}
For all $2 \leq \ell \leq 4$, we have that \[|\calC(\ell)| = O_k(N^{\ell(2k+1)/4}).\]
\end{lem}

\begin{proof}
We will use two slightly different approaches.
The first one uses the equivalence described in Equations~\eqref{eq:equiva} and~\eqref{eq:equivb}, namely that two sumsets are equal if and only if the sums of the minimal elements and the sumsets of the distance sets are equal.

Suppose we have four pairwise distinct $k$-sets $A,B,C,D\subset [N]$ such that $A+B=C+D$, which by the previously mentioned equivalence means that we also have $a+b=c+d$ and $A'+B'=C'+D'$\footnote{Recall that for a set $A$, $a$ denotes its minimal element, and $A'$ the set $A-a$.}.
Now, since $0$ will be contained in each of the sets $A',B',C'$ and $D'$, we see that all four of them will be contained in $A'+B'=C'+D'$.
Hence, after fixing $A'$ and $B'$, the potential elements of $C'$ and $D'$ must be chosen from the at most $k^2$ elements in $A'+B'$, and hence there are at most $O_k(N^{2k-2})$ choices for the quadruple $(A',B',C',D')$.
Since we have the relation $a+b=c+d$, at most three of these elements can be chosen freely, which results in \[|\calC(4)|=O_k(N^{2k+1}).\]

A similar argument works for the case $\ell=2$, by noting that this implies a sumset equality of the form $A+A=A+B$ or $A+A=B+B$, with $A\neq B$, and so the elements of $B'$ have to be chosen from $A'+A'$.
So there are $O_k(N^{k-1})$ choices for the pair $(A',B')$, and since the minimal elements satisfy the equality $2a=a+b$ or $2a=2b$, and hence $a=b$, only one of them can be chosen freely, which results in the upper bound \[|\calC(2)| = O_k(N^{k}).\]

For $\ell=3$, there are two possible types of sumset equality, namely \[A+A=B+C\] or \[A+B=A+C.\]
The first case can be handled the same way as before, noting that one can choose at most two of the three minimal elements $a,b,c$ freely, and so there are at most $O_k(N^{k+1})$ such bad solutions.
The second type requires us to make a slightly different argument.
Since $B\neq C$, we see that $B\setminus C \neq \emptyset$.
Let us for now assume that there is a unique $b\in B\setminus C$, and let $A=\{a_1,\dots,a_k\}$ and $C=\{c_1,\dots,c_k\}$.
Since $A+B=A+C$, we have that $A+b \subseteq A+C$, and hence there exist functions $\pi,\tau \colon[k]\to [k]$ such that for all $i\in[k]$, \[a_i + b = a_{\pi(i)}+c_{\tau(i)}.\]
Furthermore, since $b \notin C$, we see that $a_i \neq a_{\pi(i)}$ for all $i\in[k]$.
Write this linear system of equations in matrix form as
\begin{equation}\label{eq:l=3}
M \cdot (a_1,\dots,a_k,b,c_1,\dots,c_k)^T = 0,
\end{equation}
then the $i$-th row of $M$ will have $1$s in the $i$-th and $(k+1)$-st column, as well as $-1$s in the $\pi(i)$-th and $(k+1+\tau(i))$-th one, and $0$s everywhere else.
We will show that $M$ has rank at least $\ceil{k/2}+1$.
Let us ignore the last $k$ columns and only focus on those corresponding to $A$ and $b$.
We start with row $1$, which is clearly nonzero.
Call any row that has its $1$ entry in a column in which a previously picked row has a $-1$ entry \emph{closed}.
For example, at the start, only row $\pi(1)$ is closed.
Proceed by successively picking a row among the non-closed ones.
Since every row contains only a single $-1$ entry, it is not hard to see that at the end of this process, we have at least $\ceil{k/2}$ linearly independent rows.
Finally, because the $(k+1)$-st column in every row vector is $1$, we can pick a single arbitrary closed row and end up with $\ceil{k/2}+1$ linearly independent rows, which implies the lower bound on the rank of $M$.
Now note that if $|B\setminus C| > 1$, we can do the same, but add more columns for the remaining elements in this set.
Since we get at least one linearly independent row per new column as well, the general lower bound will hence be \[\rk(M) \geq \ceil{k/2}+|B\setminus C|.\]
Hence, at most \[2k + |B\setminus C| - \left(\ceil*{\frac{k}{2}}+|B\setminus C|\right) = k+\floor*{\frac{k}{2}} \leq 3k/2\] of the elements in $A,C$ and $B\setminus C$ can be chosen freely to satisfy the equation \eqref{eq:l=3}. 
Since elements from $B\cap C$ have to be in $C$, the same is true (up to maybe some factor only depending on $k$) for $A,B,C$. 
Since there are only $O_k(1)$ possible matrices $M$ which may lead to an equation of the form $A+B=A+C$, we conclude that \[|\calC(3)| = O_k(N^{3k/2}).\]
This completes the proof.
\end{proof}

We are now ready to prove the $1$-statement of Theorem~\ref{thm:randomsidon}.

\begin{prop}\label{prop:randomsidon_1}
If $\calA\in\calS(N,k,p)$ is a random family with $p=o(N^{-(2k+1)/4})$, then \[\lim_{N\to\infty}\Pr(\calA\text{ is a Sidon system})=1.\]
\end{prop}

\begin{proof}
By linearity of expectation and Markov's inequality, we see that
\begin{equation} \label{eq:randomsidon_1}
\Pr (X \geq 1) \leq \Exp (X) = \sum_{(A,B)\in\calC} \Exp (I_{A,B}) = \sum_{\ell=2}^4\sum_{(A,B)\in\calC(\ell)} \Exp(I_{A,B})= \sum_{\ell=2}^4 |\calC(\ell)|p^\ell.
\end{equation}
Since $|\calC(\ell)|=O_k(N^{\ell(2k+1)/4})$ by Lemma~\ref{lem:CLupper} and $p=o(N^{-(2k+1)/4})$, we thus have \[\Pr(X \geq 1) = o(1),\] proving the statement.
\end{proof}

We will next prove the $0$-statement.
Let \[\calC' = \{(A,B) = ((A_1,A_2),(B_1,B_2)) \in \calC(4) : A_1'\neq A_2',\, A_1'=B_1' \text{ and } A_2'=B_2'\}.\]
This implies in particular that for any $(A,B)\in\calC'$, the minimal elements $a_i$, $b_i$ satisfy $a_i \neq b_i$ for $i=1,2$.
Let \[Y = \sum_{(A,B)\in\calC'}I_{A,B},\] then clearly \[\Pr(X = 0) \leq \Pr(Y = 0),\] and so it suffices to show the $0$-statement for $\calC'$.
Let $\calD$ denote the family of elements $((A,B),(C,D))$ in $\calC'\times\calC'$ such that $(A,B)\neq(C,D)$ and \[\{A_1,A_2,B_1,B_2\} \cap \{C_1,C_2,D_1,D_2\} \neq \emptyset,\] that is, it contains the pairs of distinct elements in $\calC'$ which share at least one $k$-set.
For $4\le \ell\le 7$ we define the families \[\calD(\ell)=\{((A,B),(C,D))\in\calD : |\{A_1,A_2,B_1,B_2, C_1,C_2,D_1,D_2\}| =\ell\}.\]

We first give a lower bound for $|\calC'|$.

\begin{lem}\label{lem:C4lower}
$|\calC'| = \Omega_k(N^{2k+1})$.
\end{lem}

\begin{proof}
We can choose $\Omega_k (N^{2(k-1)})$ different pairs $A_1'\neq A_2'$ of sets in ${[N]\choose k}_0$ and, for every such pair, we can choose $\Omega (N^3)$ elements $a_1<b_1<b_2<a_2\in [N]$ with $a_1+a_2=b_1+b_2$ and $(a_1+A_1'),(a_2+A_2'),(b_1+A_1'),(b_2+A_2')\in {[N]\choose k}$ which form elements in $\calC'$, and so the statement follows.
\end{proof}

We next give upper bounds for $|\calD(\ell)|$ along the same lines as for the upper bounds of $|\calC (\ell)|$ in Lemma~\ref{lem:CLupper}. 

\begin{lem}\label{lem:DLupper}
For $4\leq\ell\leq 7$, we have that \[|\calD(\ell)| = \begin{cases}0, &\text{if } \ell=4,5\\O_k(N^{3k+1}), &\text{if }\ell=6\\O_k(N^{3k+2}), &\text{if }\ell=7 \end{cases}\]
\end{lem}

\begin{proof}
We first prove that $\calD(4)$ and $\calD(5)$ have to be empty.
Indeed, for $\ell=5$, three of the sets in $(C,D)$ will be fixed by $(A,B)$, and so, by the definition of $\calC'$, the last one will be as well.

For $\ell=4$, in order to get a nontrivial solution, $(C,D)$ must define an equation of the form $A_1+B_2=B_1+A_2$ and hence $a_1+b_2=b_1+a_2$, which together with $a_1+a_2=b_1+b_2$ implies $a_1=b_1$ and $a_2=b_2$, a contradiction to the definition of $\calC'$.

Suppose $\ell =7$.
For each pair $(A,B)$ there is a fixed set in the pair $(C,D)$, say e.g. $C_1$.
Since $C'_1,C'_2,D'_1,D'_2\in {[N]\choose k}_0$, there are $O_k(N^{k-1})$ choices for $C'_2$ and only $O_k(1)$ choices of $D'_1$ and $D'_2$ afterwards, since they have to be chosen from the elements of $C_1'+C_2'$. 
Since $c_1$ is fixed and we have $c_1+c_2=d_1+d_2$, there are $O(N^2)$ choices for $c_2,d_1,d_2$. 
Summarizing, \[|\calD (7)|= |\calC'|O_k(N^{k+1})=O_k(N^{3k+2}).\]

For $\ell =6$, two of the sets in the pair $(C,D)$ are fixed by the pair $(A,B)$. 
By repeating the above reasoning, we have at most $O(N)$ choices for $c_1,c_2,d_1,d_2$ giving \[|\calD (6)|= |\calC'|O_k(N^{k})=O_k(N^{3k+1}).\]
This completes the proof.
\end{proof}

We are now ready to prove the $0$-statement of Theorem~\ref{thm:randomsidon}.

\begin{prop}\label{prop:randomsidon_0}
If $\calA\in\calS(N,k,p)$ is a random family with $p=\omega(N^{-(2k+1)/4})$, then \[\lim_{N\to\infty}\Pr(\calA\text{ is a Sidon system})=0.\]
\end{prop}

\begin{proof}
By the Janson inequality (see e.g. Theorem 1.1 in Chapter 8 of~\cite{AS08}) we obtain
\begin{equation}\label{eq:randomsidon_0}
\Pr (Y=0) \leq \prod_{(A,B)\in\calC'} \Pr (I_{A,B}=0) \cdot \exp(\Delta) = (1-p^4)^{|\calC'|}\exp(\Delta),
\end{equation}
where 
\begin{equation}\label{eq:randomsidon_delta} 
\begin{split}
\Delta &= \sum_{((A,B),(C,D))\in\calD} \Pr (I_{A,B}I_{C,D}=1)\\
&= \sum_{\ell=5}^7 \sum_{((A,B),(C,D))\in\calD(\ell)} \Pr(I_{A,B}I_{C,D}=1)\\
&= \sum_{\ell=4}^7 |\calD(\ell)|p^\ell.
\end{split}
\end{equation}

By inserting the upper bounds from Lemma~\ref{lem:DLupper} into~\eqref{eq:randomsidon_delta} we get \[\Delta = \sum_{\ell=4}^7 |\calD(\ell)|p^\ell = O_k(N^{3k+2}p^7+N^{3k+1}p^6),\] and for $p = o(N^{-(2k+1)/4})$, it is easy to check that \[N^{3k+1}p^6 = \omega(N^{3k+2}p^7) \quad\text{and}\quad N^{3k+1}p^6 = o(1),\]
which implies \[\Delta = o(1).\]
Using this and the lower bound on $|\calC'|$ obtained in Lemma~\ref{lem:C4lower}, \eqref{eq:randomsidon_0} therefore gives 
\begin{align*}
\Pr (X=0) 
&\leq \Pr(Y=0) \\
&\leq (1-p^4)^{|\calC'|}\exp(\Delta) \\
&\leq \exp(-\Omega_k(p^4 N^{(2k+1)})+o(1)) \\
&= \exp(-\omega(1)).
\end{align*}
This completes the proof.
\end{proof}

\section{Concluding remarks}\label{sec:final}

\subsection{Asymptotically sharp lower bounds for $F_k(N)$}

The most begging question left open by the current work is whether Theorem~\ref{thm:lower} is asymptotically sharp for $k\geq 4$.
Since for $k\geq 3$, translations cannot generate a significant number of new sets, this is essentially equivalent to saying that one can remove $o(N^{k-1})$ sets from $\binom{[N-1]}{k}_0$ such that the resulting family is a Sidon system.
If we consider the family $\binom{[N-1]}{k}_0+\binom{[N-1]}{k}_0$, then a randomly chosen element $S$ will asymptotically almost surely have cardinality $k^2$, so it is reasonable to assume that sumsets of this cardinality are the most important case to consider.
However, while sumsets of this type have only one representation in the case $k=3$, this will in general not be true anymore for larger values of $k$.
To see this, consider for instance the case $k=4$, and let $a,b,c,d$ be integers such that 
\begin{equation*}
S=\{0,a\}+\{0,b\}+\{0,c\}+\{0,d\} \quad\text{and}\quad |S|=16.
\end{equation*}
Then, in general, we have three different representations for $S$ as a sumset of two $4$-sets, namely by pairing $\{0,a\}$ with one of the remaining three $2$-sets, and pairing the other two.
Similar constructions can be done for any $k$ that is composite.
Numerical experiments for moderate values of $N$ and $k$ ($N=100$ for $k=4$, $N=60$ for $k=5$) suggest that these might be the only instances of sumsets of cardinality $k^2$ that violate the Sidon property.
However, note that this does in general not refute the statement in the beginning of this section, since $k$-sets constructed in such a way will always obey some linear equations.
In the case $k=4$ for example, the largest element of a $4$-set always has to be the sum of the other two nonzero elements, and hence there are only $N^2=o(N^3)$ such sets, and one can remove them from $\binom{[N-1]}{4}_0$ without affecting the asymptotic density.

\subsection{$B_h[g]$ systems}

It is also possible to further generalize the definition of a Sidon system, in the same way that Sidon sets can be generalized to so called $B_h[g]$ sets.
For a family $\calA$ of integer subsets, a set of integers $C$, and an integer $h\geq2$, let $r_{h\calA}(C)$ denote the number of different multisets $\{A_1,A_2,\dots,A_h\}$, $A_i\in\calA$ such that $A_1+A_2+\dots+A_h = C$.
A \emph{$B_h[g]$ system} is a family $\calA$ of integer subsets such that $r_{h\calA}(C)\leq g$ for all sets $C\subseteq\Z$.
So a Sidon system is a $B_2[1]$ system.
One can now define $F_{k,g,h}(N)$ as the largest cardinality of a $B_h[g]$ system $\calA \subseteq \binom{[N]}{k}$.
We consider $h=2$, and write $F_{k,g,2}(N)=F_{k,g}(N)$.
We prove the following upper bound.

\begin{thm}\label{thm:gsysupper}
For $k,g \geq 2$, \[F_{k,g}(N) = O_k(\sqrt{g} N^{k-1/2}).\]
\end{thm}

\begin{proof}
There are $O_k(N^{2k-1})$ sumsets of the form $A+B$, with $A,B\in \binom{[N]}{k}$.
Indeed, since any fixed sumset $A+B$ with $A,B\in \binom{[N]}{k}$ is essentially a translation of a sumset of two sets in $\binom{[N]}{k}_0$, there are at most $O_k(N^{2k-1})$ of them.
Now, if $\calA\subseteq \binom{[N]}{k}$ is a $B_2[g]$ system, then \[\binom{|\calA|+1}{2} = \sum_{S\subseteq\Z} r_{2\calA}(S) = O_k(gN^{2k-1}).\]
Simplifying this gives the upper bound.
\end{proof}

\begin{rmk}
This can be made more precise for specific values of $k$.
For instance, there are exactly $N(N-1)^2/2$ sumsets in the case $k=2$, which gives a bound $F_{2,g}(N)\leq \sqrt{g}N^{3/2}$.
\end{rmk}

At first sight this could seem like a rather weak statement, since the results of Sections~\ref{sec:ub} and~\ref{sec:lower} give $F_{k,1}(N)=\Theta_k(N^{k-1})$.
However, we will see that for any $g\geq 2$, $\sqrt{g}N^{k-1/2}$ is indeed the right order for $F_{k,g}(N)$.
More specifically, we get the following lower bound.

\begin{thm}\label{thm:gsyslower}
Let $k \geq 2$, then \[F_{k,g}(N) = \Omega_k(\sqrt{g}N^{k-1/2}).\]
\end{thm}

\begin{proof}
Let $A \subseteq \{1,2,\dots,N/2\}$ be a $B_2[\floor{g/2}]$ set such that $|A| = \Theta(\sqrt{gN})$, which is well known to exist (c.f.~\cite{OBryant04}) and let $\calI\subseteq \binom{[N/2]}{k}_0$ be a Sidon system.
We can use the constructions from Section~\ref{sec:lower} and see that $|\calI| = \Omega_k(N^{k-1})$.
We will show that the family \[\calA = \{a+I : a\in A,\, I\in\calI\}\] is a $B_2[g]$ system.
Suppose \[(a+I)+(b+J) = (c+L)+(d+M),\] then by the arguments stated in Section~\ref{sec:not}, we must have \[a+b=c+d \quad\text{and}\quad I+J=L+M.\]
Since $A$ is a $B_2[\floor{g/2}]$ set, $a+b$ has at most $\floor{g/2}$ representations, and since $\calI$ is a Sidon system, the latter implies $\{I,J\}=\{L,M\}$.
Hence there are at most $g$ representations for this sumset.
Since we clearly have $|\calA| = |A||\calI| = \Omega_k(\sqrt{g}N^{k-1/2})$, this completes the proof.
\end{proof}

\subsection{Results for $h$-fold sumsets}

Another aspect to consider is to generalize Theorems~\ref{thm:upper}, \ref{thm:lower} and~\ref{thm:randomsidon} to the case of $h$-fold sumsets for some arbitrary fixed $h\geq 2$.
We are able to prove the following generalization of the $0$-statement of Theorem~\ref{thm:randomsidon}.

\begin{prop}\label{prop:randombh1}
Let $N,k,h$ be integers, $k\geq 2$ and $0\leq p\leq 1$.
Let $\calA \in \calS(N,k,p)$ be a random system and define \[p_0(N,k,h) = p_0 = N^{-\frac{hk+1}{h+2}}.\]
If $p=\omega(p_0)$, then $\calA$ is asymptotically almost surely not a $B_h[1]$ system.
\end{prop}

\begin{proof} The proof is an adaptation of the one in Theorem \ref{thm:randomsidon} for the $0$--statement in the case $h=2$. As in that proof it suffices to exhibit the occurrence of a particular class of $h$--tuples violating the Sidon property.

We recall that, with the above notation, the equation
$$
A_1+A_2+\cdots+A_h=B_1+B_2+\cdots+B_h
$$
violating the Sidon condition is equivalent to
$$
A'_1+\cdots +A'_h=B'_1+\cdots +B'_h\; \text{and} \; a_1+\cdots +a_h=b_1+\cdots +b_h.
$$
Let $Z$ be a random system of $k$--sets. Consider the set
$$
F=\{A=(A_1,\ldots ,A_h): A_i\in Z, A_i\preceq A_{i+1}, i=1,\ldots ,h-1\},
$$
of ordered $h$--tuples of sets in $Z$. Denote by $G$ the family of ordered pairs of distinct $h$--tuples $(A,B)\in F\times F$ satisfying the following properties:

\begin{enumerate}
\item[(i)] $A_i=a_i+A'_i$ and $B_i=b_i+A'_i$, $i=1,\ldots h$.
\item[(ii)] $a_i=b_i$ for all but two subscripts in $[h]$, the $a_i$'s are pairwise distinct and the $b_i$'s are pairwise distinct, and
\item[(iii)] $\sum_{i}a_i=\sum_{i}b_i$.
\end{enumerate}
Thus $G$ consists of pairs violating the Sidon sets with $h+2$ sets in total. 

Let $Y$ be the random variable counting the number of pairs in $G$. The Proposition will be proved if we show that 
$$
\lim_{N\to\infty} \Pr (Y=0)=0.
$$
There are $\Omega (N^{h(k-1)})$ choices for the $A'_1,\ldots ,A'_h$ and $\Omega(N^{h+1})$ for distinct integers  $a_1,a_2,\ldots ,a_h,b_1,b_2$ satisfying
$$
a_1+a_2+a_3+\cdots +a_h=b_1+b_2+a_3+\cdots +a_h.
$$
It follows that  
$$
|G| = \Omega(N^{hk+1}).
$$
By the Janson inequality,
\begin{equation}
\label{eq:jansonineq}
\Pr (Y=0) \leq \prod_{(A,B)\in G} \Pr (I_{A,B}=0) \cdot \exp(\Delta) = (1-p^{h+2})^{|G|}\exp(\Delta),
\end{equation}
where, by denoting by $K$ the subset  consisting of distinct pairs $((A,B),(C,D))\in G\times G$ such that $\{A_1,\dots,A_h,B_1,\dots,B_h\}\cap \{C_1,\dots,C_h,D_1,\dots,D_h\}\neq \emptyset$ and by $K(m)$ the pairs of $K$ which have $m$ subsets in total, 
\begin{equation}\label{eq:jansondelta} 
\begin{split}
\Delta &= \sum_{((A,B),(C,D))\in K} \Pr (I_{A,B}I_{C,D}=1)\\
&= \sum_{m=h+2}^{2h+3} \sum_{((A,B),(C,D))\in K(m)} \Pr(I_{A,B}I_{C,D}=1)\\
&= \sum_{m=h+2}^{2h+3} |K(m)|p^m.
\end{split}
\end{equation}

We will next upper bound the cardinalities of the sets $K(m)$.

\begin{claim}
\label{lem:kupper}
We have \[|K(m)| = \begin{cases} O_{h,k}(N^{\floor{(m-h)/2}(k-1)-k+m-h+hk-1}) &\text{if } h+4 \leq m \leq 2h+3, \\ 0 & \text{otherwise}.  \end{cases}\]
\end{claim}

\begin{proof} By the definition of $G$ we cannot have distinct pairs $(A,B), (C,D)\in G$ having at most $h+3$ sets in total, which shows that $K(h+2)=K(h+3)=\emptyset$. 

Suppose that $h+4\le m\le 2h+3$.
For a fixed $(A,B) \in G$, we first note that only $m-h-2$ of the sets defined by a potential $(C,D)$ can still be chosen freely.
So the same is true for the minimal elements $c_i,d_i$.
Furthermore, we also have the equation $\sum c_i = \sum d_i$, which results in another non-redundant restriction, and hence there are at most $m-h-3$ choices for the minimal elements of the pair $(C,D)$.
We now consider the distance sets.
By the definition of $G$, there are exactly $h$ pairwise distinct distance sets $C_1',\dots,C_h'$.
Since at most $m-h-2$ of the sets defined by $(C,D)$ are not determined by $(A,B)$, we thus have at most \[\floor*{\frac{m-h-2}{2}} = \floor*{\frac{m-h}{2}}-1\] undetermined distance sets.
There are at most $O(N^{(\floor{(m-h)/2}-1)(k-1)})$ of such sumsets, and so 
\begin{align*}
K(m) &= |G|\cdot O(N^{(\floor{(m-h)/2}-1)(k-1)}) \cdot O(N^{m-h-3})\\
&= O(N^{\floor{(m-h)/2}(k-1)-k+m-h+hk-1}).
\end{align*}
\end{proof}

It follows from the above claim that, if \[p = o(N^{-\frac{\floor{(m-h)/2}(k-1)-k+m-h+hk-1}{m}}),\] then 
\begin{equation}
\label{eq:Kvanishing}
|K(m)|p^m = o(1), \; h+2\le m\le 2h+3.
\end{equation}
We note that, for  $k,h\geq 2$ and $h+4\leq m \leq 2h+3$, we have
$$
\frac{(m-h)(k-1)/2-k+m-h+hk-1}{m} < \frac{hk+1}{h+2}.
$$
Hence, for any any $p$ such that \[p=o(N^{-\frac{\floor{(m-h)/2}(k-1)-k+m-h+hk-1}{m}}) \quad \text{and} \quad p=\omega(N^{-\frac{hk+1}{h+2}}),\] we see that \eqref{eq:Kvanishing} holds, and furthermore when looking at \eqref{eq:jansonineq} we get
\begin{align*}
\Pr(Y=0)
&\leq (1-p^{h+2})^{|G|}\exp(\Delta)\\
&\leq \exp(-|G|p^{h+2}+\Delta)\\
&= \exp(-\omega(1)+o(1)),
\end{align*} 
and so the family is asymptotically almost surely not a $B_h(1)$ system for $p$ in this range.
But this property is clearly monotone in $p$, and hence we get the $0$-statement for all $p=\omega(N^{-\frac{hk+1}{h+2}})$.
\end{proof}

Recall that the threshold in Theorem~\ref{thm:randomsidon} essentially resulted from the fact that the most important case was that of four pairwise distinct sets $A,B,C,D$ such that \[A+B=C+D.\]
Especially in light of the results obtained by Godbole et al. in~\cite{GJLR99}, the initial assumption might be that this corresponds to the case $2h$ for general $h$, which would lead to a threshold at \[p_1 = N^{-\frac{h(k+1)-1}{2h}}.\]
Proposition~\ref{prop:randombh1} in fact shows that the case of $h+2$ pairwise distinct sets is more important, since one can easily check that $p_0 \leq p_1$, with equality only if $h=2$.
Further evidence that the $p_0$ in Proposition~\ref{prop:randombh1} might be the correct threshold is given by the fact that it is easy to check that the $1$-statement holds in the special case of $k=2$ for any $h\leq 4$.
Note that inserting $k=1$ into the threshold $p_0$ gives something strictly weaker than the result by Godbole et al. for all $h>2$, and so the consistency in the $h=2$ case seems to be more of a coincidence.

\appendix
\section{Nontrivial Sumset Equalities for $3$-Sets}\label{sec:3solutions}

In this section we will show that in the case of $3$-sets, there are only few nontrivial sumset equalities.
Moreover, from the proof of the following theorem, it will be possible to extract the specific solutions mentioned in the proof of Proposition~\ref{prop:k3}.

\begin{thm}\label{thm:3solutions}
\[\left|\{(X,Y,V,W)\in \binom{[N]}{3}_0^4 : X+Y=V+W \text{ and } \{X,Y\}\neq\{V,W\}\}\right| = O(N).\]
\end{thm}

\begin{proof}
For any set $X=\{0<x<x'\}\in \binom{[N]}{k}_0$, let $\widebar{X}$ denote the set \[\widebar{X} = x'-X = \{0<x'-x<x'\}.\]
We will refer to $\widebar{X}$ as the \emph{dual} of $X$.
Let $X=\{0<x<x'\},Y,V,W\in\binom{[N]}{k}_0$ such that $X+Y=V+W$.
Without loss of generality, we can assume that $X\preceq Y$, $X\preceq V$, and $V\preceq W$.
Furthermore, $X=V$ only if $Y\prec W$.
We begin by stating all orderings satisfying the above relations, but not having any solutions.
They are
\begin{gather*}
X=Y\prec V=W,\quad X=Y=V\prec W,\quad X\prec Y=V=W,\\
X=Y\prec V\prec W,\quad X\prec Y\prec V=W,\quad X\prec V=W\prec Y,\\
X\prec V=Y\prec W,\quad X\prec V\prec Y=W,\quad X\prec Y\prec V\prec W.    
\end{gather*}
We will quickly go through all of them and see that they are not possible.

\hfill

$\bm{X=Y\prec V=W.}$
This implies $x=v$ as they are the smallest element on either side, and also $2x'=2v'$ as the largest element on either side, which implies $x'=v'$, and hence $X=Y=V=W$, a contradiction.

\hfill

$\bm{X=Y=V\prec W.}$
This implies $2x'=x'+w'$ and hence $x'=w'$.
Then we must have $w>x$, but now $x'+w$ is strictly larger than $x'+x$ and hence cannot be matched by any element of $X+X$, a contradiction.
Note that $w>x$ also implies $x'-x>w'-w$, that is, $\widebar{W}\prec \widebar{X}$, and so the case $X\prec Y=V=W$ is also not possible.

\hfill

$\bm{X=Y\prec V\prec W.}$
We have $v \leq w$, and hence $x=v$ as the smallest element on either side.
This implies $x'<v'$ which in turn means $x'>w'$ since $2x'=v'+w'$.
But then $v'+w'>v'+w>v'>w'>w\geq v>0$ is a chain in $V+W$ of at least $6$ elements, while $|X+X|\leq 6$.
Hence we must have $v=w$ which implies $w'>v'$, a contradiction.

\hfill 

$\bm{X\prec Y\prec V=W.}$
We have $x\leq y\leq v$ and hence $x=y=v$ because the smallest elements coincide.
But then $x'<y'<v'$ which contradicts $x'+y'=2v'$.

\hfill

$\bm{X\prec V=W\prec Y.}$
We have $x\leq v\leq y$ and hence $x=v\leq y$, which implies $x'<v'$, and hence $y'>v'>x'$.
But then $x+y'$ is strictly larger than $v+v'$ and hence cannot be matched by any element of $V+V$.

\hfill

$\bm{X\prec V=Y\prec W.}$
Since $x'+y'=y'+w'$ we know that $x'=w'$.
Furthermore, $x\leq y\leq w$, and hence $x=y$ as smallest elements.
This implies $x'<y'$, which in turn means $w'<y'$, and hence $w>y$.
But then $y'+w$ is strictly larger than $y'+x$ and $x'+y$ and hence cannot be matched by any element of $X+Y$.

\hfill

$\bm{X\prec V\prec Y=W.}$
We have $x'+y'=v'+y'$ and hence $x'=v'$.
Furthermore $x\leq v\leq y$, which implies $x=v$ and hence $x'<v'$, a contradiction.

\hfill

$\bm{X\prec Y\prec V\prec W.}$
Since $x\leq y\leq v\leq w$, we have $x=y=v\leq w$ as smallest elements, which implies $x'<y'<v'$.
But now $v'+w$ will be strictly larger than both $x'+y$ and $y'+x$ and hence cannot be matched by any element of $X+Y$.

\hfill

This settles the cases without solutions.
The remaining three cases are \[X=V\prec Y\prec W, \quad X\prec V\prec Y\prec W, \quad X\prec V\prec W\prec Y,\] each of which will actually have solutions, and the analysis is more involved than in the previous cases.

\hfill

$\bm{X=V\prec Y\prec W.}$
Before splitting this case into further subcases, we make some general assertions.
Since $x'+y'=x'+w'$, we know that $y'=w'$, and hence $w>y$ (in particular, $w>x$).
This means that $x'+w>x'+y$, and hence $x'+w\in\{y',x+y'\}$.
Since $w>x$, this implies in particular that $y'>x'$.
It also tells us that $x+y'>x'+y$.
Finally, with the same arguments we also see that $x+w\in\{x',y',x'+y\}$ and $w\in\{x',x+y,x'+y\}$.
Let us now consider specific subcases.

\textbf{Case a) $\bm{x'+w=y'}$.}
This implies that $y'>x'+y$.
We get the following diagrams that represent the sumsets. 
In this and all following diagrams, we will always omit the unique greatest element $x'+y'$, as well as the unique least element $0$.
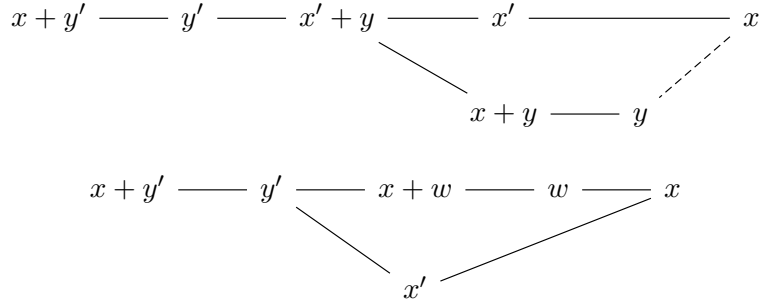
\begin{tikzccd}
x+y'\edge{r} & y'\edge{r} & x'+y\edge{r}\edge{rd} & x'\edge{rr} & & x\edge{dashed,ld} \\
& & & x+y\edge{r} & y
\end{tikzccd}
\begin{tikzccd}
x+y' \edge{r} & y'\edge{r}\edge{rd} & x+w\edge{r} & w\edge{r} & x\edge{lld}\\
& & x'
\end{tikzccd}
We see that there is at least one element strictly between $y'$ and $x'$ in the top diagram, and so it must hold that $x+w>x'$, in particular $x+w=x'+y$, so in particular $x'+y>w$, and hence $w\in\{x',x+y\}$.
We also see that $w=x'$ if and only if $x=y$.

\textbf{Case a)i. $\bm{w=x'}$.} 
Then $x=y$ and we have a chain of exactly $7$ elements in the bottom diagram.
This implies that $x'=2x$ and we get the following solution: $x$ can be freely chosen from $[1,(N-1)/4]$, $y=v=x$, $x'=v'=w=2x$, and $y'=w'=4x$.

\textbf{Case a)ii. $\bm{w=x+y\neq x'}$.} 
In particular, we have $y>x$.
We see that $x'+y=x+w=2x+y$ and hence $x'=2x$.
Also, since $w>y>x$, we must have that $x'=y$, otherwise it cannot be matched.
This leads to the solution $v=x$, $x'=y=v'=2x$, $w=3x$, and $y'=w'=5x$, where $x$ can be chosen freely from the positive integers smaller than $(N-1)/5$.

\textbf{Case b) $\bm{x'+w=x+y'}$.}
This does not give us enough new information, so we go directly into three further subcases, depending on the assignment of $x+w$.
Recall that $x+w\in\{x',y',x'+y\}$, so there are three different cases to consider.

\textbf{Case b)i. $\bm{x+w=x'}$.} 
This means that $x'$ is strictly larger than $x+y$, but strictly smaller than $y'$.
We get the following diagrams. 
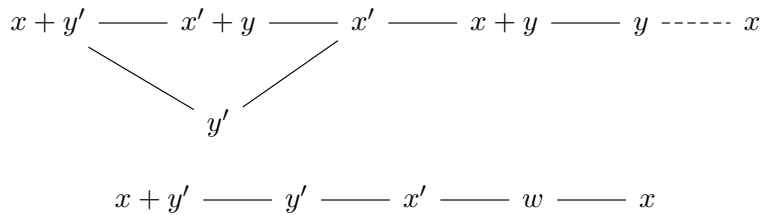
\begin{tikzccd}
x+y'\edge{r}\edge{rd} & x'+y\edge{r} & x'\edge{r}\edge{ld} & x+y\edge{r} & y\edge{dashed,r} & x\\
& y'
\end{tikzccd}
\begin{tikzccd}
x+y'\edge{r} & y'\edge{r} & x'\edge{r} & w\edge{r} & x
\end{tikzccd}
This directly implies that we must have $y=x$, since there is only one element strictly between $x'$ and $x$ in the bottom diagram.
We thus have $w=x+y=2x$ and $x'+y=y'$.
Solving this, we get the solution $y=v=x$, $w=2x$, $x'=v'=3x$, and $y'=w'=4x$, where $x$ can be chosen freely among the positive integers smaller than $(N-1)/4$.

\textbf{Case b)ii. $\bm{x+w=y'}$.} 
In particular, this means that $y'>x+y$.
We get the following diagrams, 
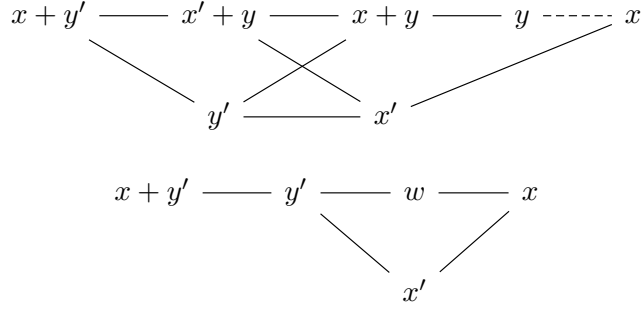
\begin{tikzccd}
x+y'\edge{r}\edge{rd} & x'+y\edge{r}\edge{rd} & x+y\edge{r}\edge{ld} & y\edge{dashed,r} & x\edge{lld}\\
& y'\edge{r} & x'
\end{tikzccd}
\begin{tikzccd}
x+y'\edge{r} & y'\edge{r}\edge{rd} & w\edge{r} & x\edge{ld}\\
& & x'
\end{tikzccd}
We see that there is no element strictly between $y'$ and $x+y'$, and hence $y'\geq x'+y$.

Suppose $y'>x'+y$.
This implies that there is an element strictly between $x'$ and $y'$, and hence we must have $w=x'+y>x'$.
So the bottom diagram will now be a chain of $7$ elements, and so we must have $y=x$ and $x+y=x'$.
This simplifies to the solution $y=v=x$, $x'=v'=2x$, $w=3x$, and $y'=w'=4x$, where $x$ can be chosen freely among the positive integers smaller than $(N-1)/4$.

Suppose now that $y'=x'+y=x+w$.
Since $y\geq x$, this implies that $w\geq x'$ with equality if and only if $x=y$.
So we have two cases.

If $w=x'$ and $x=y$, then we must have $x'=x+y=2x$, and we get the solution $y=v=x$, $x'=v'=w=2x$, and $y'=w'=3x$, where $x$ can be chosen among the positive integers smaller than $(N-1)/3$.

If $w>x'$ and $y>x$, then this implies that $x+y=w$ and $y=x'$, which leads to the solution $v=x$, $x'=y=v'=2x$, $w=3x$, and $y'=w'=4x$, with $x$ able to be chosen freely from the positive integers smaller than $(N-1)/4$.

\textbf{Case b)iii. $\bm{x+w=x'+y\neq y'}$.}
Since $y\geq x$ we have $w\geq x'$ and $y=x$ if and only if $w=x'$. 
So since $x'+y\neq y'$, the bottom side of the diagram is now a chain of $7$ or $8$ elements (we don't know whether $w=x'$).
Specifically we get the following diagrams. 
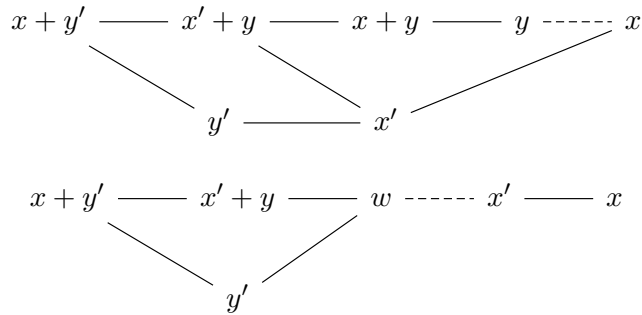
\begin{tikzccd}
x+y'\edge{r}\edge{rd} & x'+y\edge{r}\edge{rd} & x+y\edge{r} & y\edge{dashed,r} & x\edge{lld}\\
& y'\edge{r} & x'
\end{tikzccd}
\begin{tikzccd}
x+y'\edge{r}\edge{rd} & x'+y\edge{r} & w\edge{ld}\edge{dashed,r} & x'\edge{r} & x\\
& y'
\end{tikzccd}
Now note that the cases $(x=y \wedge x+y=x')$ and $x<y$ lead to $x'=2x$:
The first is trivial, for the second note that $x<y$ implies $y=x'$ and $w=x+y$, and hence $2x+y=x+w=x'+y$.
But then $y' = x'+w-x = w+x$, a contradiction.
So we must have $x=y$ and $x+y=y'=2x$, which implies $w=x'$.
We get the solution $y=v=x$, $x'=v'=w=3x/2$, and $y'=w'=2x$, where $x$ can be chosen freely among positive even integers smaller than $(N-1)/2$.

\hfill

$\bm{X\prec V\prec Y\prec W.}$
We start by observing some general relations.
Since $x\leq v\leq y\leq w$, we must have $x=v\leq y\leq w$ as the smallest elements on each side.
This implies $x'<v'$ and since $x'+y'=v'+w'$ we thus have $y'>w'$.
But then we must have $w>y$, so in particular $w>x$.
Then $v'+w$ is strictly larger than $x'+y$ and hence $v'+w\in\{y',x+y'\}$.
Since $w>x$, this implies that $y'>v'$, and hence $w'>x'$.
But then $x+w'<x+y'$, and hence we must have $v'+w=x+y'$, since otherwise $x+y'$ could not be matched.

This already determines the ordering of the duals.
First, note that \[x'-x = x'-(v'+w-y') = (v'+w'-y')-(v'+w-y') = w'-w,\] and since $w'>x'$, this implies $\widebar{X}\prec \widebar{W}$.
Furthermore \[w'-w = w'-(x+y'-v') = (w'-y') + (v'-v),\] and since $w'<y'$, this implies $v'-v>w'-w$ and hence $\widebar{W}\prec \widebar{V}$.
Finally, \[v'-v = y'+x-w-v = y'-w < y'-y,\] and so $\widebar{V}\prec \widebar{Y}$.
So by identifying $X$ and $Y$ with their own respective duals, $\widebar{W}$ with $V$, $\widebar{V}$ with $W$, this corresponds to a case of the form $X\prec V\prec W\prec Y$, and so it suffices to check these.

\hfill

$\bm{X\prec V\prec W\prec Y.}$
We will again first state some universally true things and then make case distinctions.
Since $x\leq v\leq w\leq y$ we must actually have $x=v\leq w\leq y$ since $x$ and $v$ are the smallest element in their respective sumset.
This implies $x'<v'$ and hence $y'>w'$.
So $x+y'>x+w'$ and hence we know that $x+y'\in\{v',v'+w\}$.

But suppose that we have $x+y'=v'$.
Then $v'>y'$ and hence we must have $w>x$ and $x'>w'$.
At the same time, there has to be an element strictly between $x'+y'$ and $x+y'$ in $X+Y$, and so in particular we must have $x'+y>y'+x$ and $x'+y=v'+w$, which implies $y>w$.
But then $x<w<\min(x',y)$ and hence it cannot be matched by any element in $X+Y$.

So we must have $x+y'=v'+w$, and so $y'\geq v'$ and hence $w'\geq x'$, in particular $y'>x'$.
Furthermore, $v'+w>x+w'$ and $x+y'\geq x'+y$ since we know that $x+y'$ is the second largest element in $V+W$.
Finally, we see that $y'\in\{v',x+w,x+w'\}$.
But in fact, we cannot have $y'=v'$, since this would imply $x'=w'$ and hence $w>x$.
But at the same time, we have $y'+x=v'+w=y'+w$ which implies $x=w$.
So $y'\neq v'$ and hence in fact $y'>v'$, which also implies $w'>x'$ and $w>x$ because of $y'+x=v'+w$.
In particular, $y>x$.
We have to consider two different cases for the assignment of $y'$.

\textbf{Case a) $\bm{y'=x+w \neq v'}$.}
This tells us in particular that $y'\leq x+y<x'+y$ and since $v'+w=y'+x=2x+w$ we have $v'=2x$.
We get the following diagrams.
\begin{tikzccd}
x+y'\edge{dashed,r} & x'+y\edge{r} & x+y\edge{dashed,r} & y'\edge{r}\edge{rd} & x'\edge{r} & x\edge{ld}\\
&&&& y
\end{tikzccd}
\begin{tikzccd}
x+y'\edge{r} & x+w'\edge{r} & y'\edge{r}\edge{rd} & v'\edge{rr} & & x\edge{ld} \\
&&& w'\edge{r} & w
\end{tikzccd}
From the bottom diagram, we infer that there is only one element strictly between $x+y'$ and $y'$, so exactly one of the inequalities in the top diagram has to be an equality.
Furthermore, since $w'$ is strictly between $y'$ and $x'$, we must have $w'=y$ and in particular $y'>y>x'$.
This also implies $w=x'$ and $x+w'=x+y$, so we must have $x+y>y'$ and $x+y'=x'+y$.
Finally, $x'+y=x+y'=v'+w=v'+x'$ implies $v'=y$.
This leads to the solution $v=x$, $x'=w=3x/2$, $y=v'=w'=2x$, and $y'=5x/2$ for any positive even integer $x$ smaller than $2(N-1)/5$.

\textbf{Case b) $\bm{y'=x+w'}$.}
This tells us in particular that $x+w'>v'$.
Furthermore, we have $x'+x+w'=x'+y'=v'+w'$ and hence $v'=x'+x$.
This implies $x'+x+w=v'+w=x+y'=2x+w'$, and hence $x'+y\geq x'+w=x+w'=y'$.
Since $y'$ is the third largest element in $V+W$, this means that $x'+y\in\{x'+y,y'\}$.

\textbf{Case b)i. $\bm{x'+y=y'}$.}
So $x'+y=x'+w$ and hence $w=y$.
Since $w'$ is strictly larger than $\max\{x',y\}$ but smaller than $y'$, we must have $w'=x+y$.
Now we must have $y=x'$, since otherwise $x'$ could not be matched, and hence $v'=w'$.
This leads to the solution $v=x$, $x'=y=w=2x$, $v'=w'=3x$, and $y'=4x$ for a positive integer $x$ smaller than $(N-1)/4$.

\textbf{Case b)ii. $\bm{x'+y=x+y'}$.}
We see that $x'+y=v'+w$ and hence $y>w$.
Furthermore, we have that $y'<x'+y$, and since $y'$ is the third largest element in $V+W$, we must have $y'\geq x+y$.

\textbf{Case b)ii.$\bm{\alpha}$ $\bm{y'=x+y}$.}
Since $v'$ and $w'$ are strictly between $y'$ and $x'$, we must have $v'=w'=y$ and in particular $y>x'$, which also implies $w=x'$.
Hence $x+w$ is strictly between $x'$ and $y'$, that is $x+w=y$.
This leads to the solution $v=x$, $x'=w=2x$, $y=v'=w'=3x$, and $y'=4x$, for any positive integer $x$ smaller than $(N-1)/4$.

\textbf{Case b)ii.$\bm{\beta}$ $\bm{y'=x+y}$.}
We get following diagrams. 
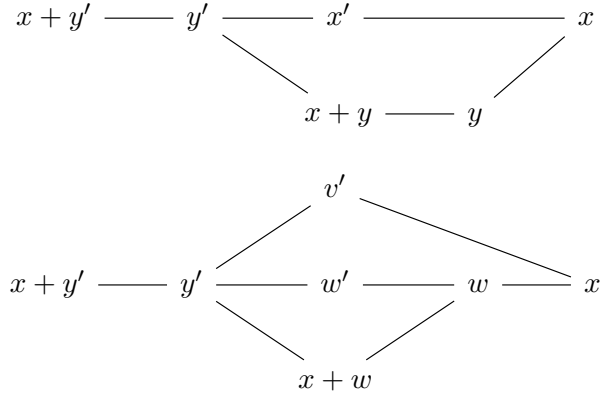
\begin{tikzccd}
x+y'\edge{r} & y'\edge{r}\edge{rd} & x'\edge{rr} & & x\edge{ld}\\
& & x+y\edge{r} & y
\end{tikzccd}
\begin{tikzccd}
& & v' \\
x+y'\edge{r} & y'\edge{ru}\edge{r}\edge{rd} & w'\edge{r} & w\edge{r}\edge{ld} & x\edge{llu}\\
& & x+w
\end{tikzccd}
Since $w',v'>x'$, we must have that $x+y>x'$ and in particular $x+y=\max\{v',w'\}$.
Furthermore, we also must have $v'>w$, since otherwise $v'$ would have to be matched to $\min(y,x')\leq x'$.
Since $w<y$, we in fact need to have $x'=w$ and in particular $x'<y$.
Thus the top diagram is a chain of $8$ elements, and hence there has to be exactly one equality on the bottom side.
Since we now know that $x+y>x+w$, we must have $x+w=y$.
Furthermore, $v'+x'=v'+w=x'+y$ and hence $v'=y$.
So since we have to match $x+y$, we must have $w'=x+y>v'$.
This leads to the solution $v=x$, $x'=w=3x$, $y=v'=4x$, $w'=5x$, and $y'=6x$, for any positive integer $x$ smaller than $(N-1)/6$.

\hfill

So there are only a constant number of different cases, each of which only has $O(N)$ different solutions.
This implies the theorem statement.
In particular, after also computing the relevant duals, we see that the sets violating the Sidon property are exactly the ones described in the proof of Proposition~\ref{prop:k3}.
\end{proof}

\noindent
\textsc{Javier Cilleruelo}: Department of Mathematics, Universidad Aut\'onoma de Madrid, 28049 Madrid, Spain.

\

\noindent
\textsc{Oriol Serra}: Department of Mathematics, Universitat Polit\`ecnica de Catalunya, 08034 Barcelona, Spain.\\
E-mail: \href{mailto:oriol.serra@upc.edu} {\texttt{oriol.serra@upc.edu}}

\

\noindent
\textsc{Maximilian W\"otzel}: Department of Mathematics, Universitat Polit\`ecnica de Catalunya and Barcelona Graduate School of Mathematics, 08034 Barcelona, Spain.\\
E-mail: \href{mailto:maximilian.wotzel@upc.edu} {\texttt{maximilian.wotzel@upc.edu}}

\blfootnote{Oriol Serra acknowledges support from the Spanish Ministerio de Econom\'ia y Competitividad projects MTM2017-82166-P and MDM-2014-0445. 
Maximilian W\"otzel acknowledges financial support from the Fondo Social Europeo and the Agencia Estatal de Investigaci\'on through the FPI grant number MDM-2014-0445-16-2 and the Spanish Ministry of Economy and Competitiveness, through the Mar\'ia de Maeztu Programme for Units of Excellence in R\&D (MDM-2014-0445), as well as through the project MTM2017-82166-P.}

\end{document}